\documentclass[a4paper, 10pt]{amsart}
\usepackage{lineno}

\usepackage{amssymb}
\usepackage{mathrsfs}
\usepackage{amsfonts}
\usepackage{hyperref}
\usepackage{amsthm}
\usepackage{amscd}
\usepackage{amsmath}
\usepackage{enumerate}

\newtheorem{thm}{Theorem}[section]
\newtheorem{lem}[thm]{Lemma}
\newtheorem{cor}[thm]{Corollary}
\newtheorem{prop}[thm]{Proposition}

\newtheorem{rem}[thm]{Remark}
\newtheorem{defi}[thm]{Definition}

\usepackage{latexsym,amssymb}
\usepackage{amsmath}

\newcommand{\Dav}{\mathsf D}

\DeclareMathOperator{\ord}{ord}

\begin{document}

\title{On the lower bounds of Davenport constant}

\author{Chao Liu}
\address{Center for Combinatorics, LPMC-TJKLC\\
 Nankai University\\
 Tianjin 300071, P.R. China}
\email{chaoliuac@gmail.com}

\urladdr{https://chaoliu.science}

\begin{abstract}
Let $G = C_{n_1} \oplus \dots \oplus C_{n_r}$ with $1 < n_1 | \dots | n_r$ be a finite abelian group. The Davenport constant $\Dav(G)$ is the smallest integer $t$ such that every sequence $S$ over $G$ of length $|S|\ge t$ has a non-empty zero-sum subsequence. It is a starting point of zero-sum theory. It has a trivial lower bound  $\Dav^*(G) = n_1 + \dots + n_r - r + 1$, which equals $\Dav(G)$ over $p$-groups. We investigate the non-dispersive sequences over groups $C_n^r$, thereby revealing the growth of $\Dav(G)-\Dav^*(G)$ over  non-$p$-groups $G = C_n^r \oplus C_{kn}$ with $n,k \ne 1$. We give a general lower bound of $\Dav(G)$ over non-$p$-groups and  show that if $G$ is an abelian group with $\exp(G)=m$ and  rank $r$, fix $m>0$ a non-prime-power, then for each $N>0$ there exists an $\varepsilon>0$ such that if $|G|/m^r<\varepsilon $, then $\Dav(G)-\Dav^*(G)>N$.
\end{abstract}



\maketitle

\section{Introduction and main results}
\bigskip

The Davenport constant has been studied since the 1960s. It naturally occurs in various branches of combinatorics, number theory, and geometry (see \cite[Chapter $5$]{GeKo2006non} and \cite{GaoGero06}). Early work on the Davenport constant and on the Erd\H{o}s-Ginzburg-Ziv Theorem  are considered as starting points of zero-sum theory.  The goal of the present paper is to provide new lower bounds for the Davenport constant.

Let $G$ be an additively written finite abelian group, say $G \cong C_{n_1}\oplus\dots\oplus C_{n_r}$, where $r=\mathsf r(G)$ is the rank of $G$ and $1< n_1|\dots|n_r$, and set $\Dav^*(G)=1+\sum_{i=1}^{r}(n_i-1)$. If $S=g_1\cdot\ldots\cdot g_\ell$ is a sequence over $G$, then $|S|=\ell$ is its length and $S$ is called a zero-sum sequence if its sum $\sigma(S)=g_1+\dots+g_\ell$ is equal to $0$. The Davenport constant $\Dav(G)$ of $G$ is the smallest integer $\ell\in\mathbb N$ such that every sequence $S$ over $G$ of length $|S|\geq\ell$ has a nonempty zero-sum subsequence. A straightforward example shows that $\Dav^*(G) \leq \Dav(G)$. Already in the 1960s it was proved that equality holds for $p$-groups and for groups having rank $r(G)\leq 2$ (see \cite[Chapter $5$]{GeKo2006non}). Here we refer to a couple of papers (\cite{re[1],re[2],re[5],re[6],re[7]}, \cite[Corollary 4.2.13]{re[3]}) of the last decade offering a growing list of groups $G$ satisfying $\Dav(G) = \Dav^*(G)$. However, it is still open whether or not equality holds for groups of rank three or for groups of the form $C_n^r$.

The first example of $\Dav(G) > \Dav^*(G)$ is due to P.C. Baayen in 1969. Let $G = C_2^{4k}\oplus C_{4k + 2}$ with $k\in\mathbb N_+$, then $\Dav(G)\ge\Dav^*(G) + 1$ (\cite[Theorem 8.1]{E1969}). We briefly introduce some works on the lower bounds of Davenport constant.
\begin{enumerate}

\item   Let $G=C_n^{(k-1)n+\rho}\oplus C_{kn}$ with  $n, k\ge 2$, $\gcd(n,k)=1$ and $0 \le \rho \le n-1$.
\begin{enumerate}
\item If $\rho\ge 1$ and $\rho\not\equiv n \pmod k$, then $\Dav(G)\ge \Dav^*(G)+\rho$.
\item If $\rho\le n-2$ and $x(n-\rho +1)\not\equiv n \, \pmod k$ for any $x\in[1,n-1]$, then $\Dav(G)\ge \Dav^*(G)+\rho +1$.     (Emde Boas and Kruyswijk, 1969)
\end{enumerate}

\item Let $ G = C_m\oplus C_n^2\oplus C_{2n}$ with $m,n\in \mathbb N_{\ge 3}$ odd and $m|n$. Then $\Dav(G)\ge\Dav^*(G) + 1$. (Geroldinger and Schneider,  \cite{geroldinger1992Dav}, 1992) 

\item   Let $G = C_2^{r - 1}\oplus C_{2k}$ with $k> 1$ odd. Then $\Dav(G) - \Dav^*(G) \ge \max \{ \log_2 r - \alpha(k) - 2k + 1,\, 0 \},$
where $\alpha(k) = i$ iff $2^{i-1} + 1 < k \le  2^i + 1$. (Mazur,  \cite{mazur1992note}, 1992)
\item Let $G = C_2^i\oplus C_{2n}^{5 - i}$ with $i\in [1,4]$ and $n\ge 3$ odd. Then $\Dav(G) \ge \Dav^*(G) + 1$. (See \cite{gao1999,geroldinger2012Dav,geroldinger1992Dav} for $i = 2$, $i = 1$ and $i \in \{3, 4\}$ separately)  
\end{enumerate}

The third result shows the growth of $\Dav(G)-\Dav^*(G)$ over $G=C_2^{r - 1}\oplus C_{2k}$ with $k$ odd. The author, Mazur, also asked if there are similar results when $k$ is even \cite{mazur1992note}.

This paper will show the growth of $\Dav(G)-\Dav^*(G)$ over non-$p$-groups $G=C_n^r\oplus C_{kn}$ with any $n,k\ne 1$ (see Theorem \ref{lzfs} and Corollary \ref{est}). We show $\Dav(G)-\Dav^*(G)$ grows at least logarithmically with respect to $r$ for a fixed $n$.  For the cases of $\gcd(k,n)\ne 1$, this is the first time to prove  $\Dav(G)=\Dav^*(G)$ false. We show that $\Dav(G)- \Dav^*(G)>0$ can happen even if the exponent of $G$ is arbitrarily large (see Remark \ref{taolun}). So Mazur's result is improved and more results are derived.

We prove the result with a new method. By Lemma \ref{nondisp2dg}, this paper connects the lower bounds of Davenport constant to the study of {\sl non-dispersive sequence}, which goes back to a conjecture of Graham reported in \cite{Er-Sz76}. A sequence $S$ over $G$ is called non-dispersive if all nonempty zero-sum subsequences of $S$ have the same length. In 1976, Erd\H{o}s and Szemer\'{e}di \cite{Er-Sz76} proved that if $S$ is a non-dispersive sequence over $C_p$ of length $p$, then $S$ takes at most two distinct values, where $p$ is a sufficiently large prime. Gao {\sl et al.} \cite{Ga-Ha-Wa10} and Grynkiewicz \cite{Gr11} independently improved  this result to all positive integers. A related question was naturally proposed by Girard \cite{Gir12} to 􏲁􏱲􏱩􏱫􏱾􏱷determine the longest length of non-dispersive sequences over any group $G$. The answer is known for group $C_2^r$ (see \cite{GZZ2015}). We   investigate   non-dispersive sequences over groups $C_n^r$ with  $n\ge 2$ (see Theorem \ref{uniquelength}), thereby improving the lower bounds of Davenport constant over $C_n^r\oplus C_{kn}$. 

We also give general lower bounds for all non-$p$-groups (see Theorem \ref{gene}) and  some other  interesting corollaries.

\section{Preliminaries}

Our notation and terminology of sequences over abelian groups is consistent with  \cite{GeKo2006non, Gr13}.

Let $n\in \mathbb N_{\ge 2}$ and set $n=pq$ for some prime $p$ and some $q\in[1,n]$. For any $\ell\in\mathbb N_{+}$, we define $\theta(\ell;p),\,\omega(\ell;n,p)$ and $\mathbf M(\ell;p,q)$ as follows through out this paper.

\begin{enumerate}[1.]
\item
\[\theta(\ell;p)=\begin{cases}
\frac{2(p^\ell-1)}{p-1}-\ell, & \text{if $p>2$}
\\
2^\ell-1-\ell, & \text{if $p=2$}  \end{cases}.\]

\item
\[\omega(\ell;n,p)=\begin{cases}
p^{\ell-1}n, & \text{if $p>2$}
\\
2^{\ell-2}n, & \text{if $p=2$}  \end{cases}.\]

\item
For any $\ell\in\mathbb N_{+}$, the set $\mathbf M(\ell;p,q)$ is constructed by a recursive algorithm:
\begin{enumerate}[(i)]
\item $\mathbf M(1;p,q) = \begin{cases}\{q, (p - 1)q\},& \text{if } p > 2\\ \{q\},& \text{if } p = 2\end{cases}.$

\item $\mathbf M(\ell + 1;p,q) = \mathbf M(\ell;p,q) \times  \mathbf A \cup \{0\}^\ell \times  \mathbf M(1;p,q)$, where $\mathbf  A = \{0,q,\dots,(p - 1)q\}$.\end{enumerate}
\end{enumerate}

\begin{rem}
We can use a direct way to construct $\mathbf M(\ell)$ for $\ell\in\mathbb N_+$,   apart from the recursive algorithm given  before. In \eqref{algorithm2}, we can let $a_s=1$ and derive that
$$ \mathbf M(\ell) =\bigcup_{t=0}^{\ell-1} \{0\}^{t} \times \mathbf M(1) \times \mathbf A^{\ell-t-1}.
$$
Hence it follows a direct way to construct the non-dispersive sequences (in Theorem \ref{uniquelength}) and zero-sum free sequences with the techniques in Lemma \ref{nondisp2dg}.

\end{rem}

Let  $|w|_n$ denote the least nonnegative residue of an integer $w$ modulo $n$. Let $| \mathbf B|$ denote the cardinality of a set $ \mathbf B$. 

The elements of $ \mathbf M(\ell;p,q)$ are $\ell$-tuples of integers. We list the elements of $ \mathbf M(\ell;p,q)$ in some fixed but arbitrary order. Then $ \mathbf M(\ell;p,q)[i,j]$ denotes the $i$-th entry of the $j$-th element of $ \mathbf M(\ell;p,q)$. We often fix some $n,p$ and $q$ before considering $\theta(\ell;p)$, $\omega(\ell;n,p)$ and $\mathbf M(\ell;p,q)$. For convenience, we might omit the parameters  ``$n$", ``$p$" and ``$q$"  when no misunderstanding is likely to occur. Thus, $\theta(\ell)$, $\omega(\ell)$ and $\mathbf M(\ell)[i, j]$ will mean $\theta(\ell;p)$,  $\omega(\ell;n,p)$ and $\mathbf M(\ell;p,q)[i, j]$ unless otherwise stated.

\begin{prop}\label{PropM}
Let $n\in \mathbb N_{\ge 2}$ and set $n=pq$ for some prime $p$ and some $q\in[1,n]$. For any $\ell\in \mathbb{N_+}$, $ \mathbf M(\ell;p,q)$ has following three properties:
\begin{enumerate}[i.]

\item $| \mathbf M(\ell)| = \theta(\ell) + \ell$.

\item
For any $1 \le a_1 < \dots < a_s\le \ell $ and any $v_i\in[1,p - 1]$ with $a_i,v_i\in\mathbb N_+$ and $i\in[1,s]$, we have
$$\sum_{j = 1}^{| \mathbf M(\ell)|} \big|\sum_{i = 1}^{s} v_i \mathbf  M(\ell) (a_i,j)\big|_n = \omega(\ell).$$

\item
$$\sum_{j = 1}^{| \mathbf M(\ell)|} \big| - \sum_{i = 1}^{s} v_i  \mathbf M(\ell )[ a_i, j]\big|_n = \sum_{j = 1}^{| \mathbf M(\ell)|} \big| \sum_{i = 1}^{s} v_i \mathbf  M(\ell) [a_i, j] \big|_n. $$

\end{enumerate}
\end{prop}

\begin{proof}
\noindent
\begin{enumerate}[1)]
\item
By the definition of $ \mathbf M(\ell)$ we can derive that $| \mathbf M(\ell + 1)| = | \mathbf M(\ell)|\cdot p + | \mathbf M(1)|,$ thus $| \mathbf M(\ell)| = \frac{(p^{\ell } - 1)| \mathbf M(1)|}{p - 1} = \theta(\ell) + \ell.$

\item
\noindent
{\bf Case 1.} $\ell = 1$.

In this case, $s=1$ and $a_1=1$. By the definitions of $ \mathbf M(1)$ and $v_1$, it is easy to infer that
\[\sum_{j=1}^{| \mathbf M(1)|}\big|v_1  \mathbf M(1)[1, j]\big|_n = \begin{cases}
p q, & \text{if $p>2$}
 \\
 q, & \text{if $p=2$}
 \end{cases}.\]\\

\noindent
{\bf Case 2.} $\ell \geq 2$ and $a_s = \ell$.

By the rules of Cartesian product and the definition of $ \mathbf M(1)$, we derive that
\[
\begin{aligned}
 \mathbf M(\ell) & =  \mathbf M(\ell-1) \times  \mathbf A \cup \{0\}^{\ell-1} \times  \mathbf M(1)
\\&=(\bigcup_{t=0}^{p-1} \mathbf M(\ell-1) \times \{tq\}) \cup \{0\}^{\ell-1} \times  \mathbf M(1).\end{aligned}
\]
Consequently,
\begin{equation}\label{mc21}
\begin{aligned}
\sum_{j=1}^{| \mathbf M(\ell)|}  \big|\sum_{i=1}^{s} v_i  \mathbf M(\ell) [a_i, j]\big|_n  =& \sum_{t=0}^{p-1} \sum_{j=1}^{| \mathbf M(\ell-1)|} \big|\sum_{i=1}^{s-1} v_i  \mathbf M(\ell-1) [a_i, j] +v_stq\big|_n\\
& + \sum_{j=1}^{| \mathbf M_1|}\big|0+v_s \mathbf  M(1)[1, j]\big|_n.
\end{aligned}
\end{equation}

Note that, for any $x \in\{0, q,\dots, (p-1)q\}$, by $v_s \in[1,p-1]$, we have $\gcd(v_s,p)=1$. Thus
\begin{equation}\label{mc22}
\sum_{t=0}^{p-1}|x + v_stq|_n = 0 + q + \dots + (p-1)q = \frac{(p-1)pq}{2}.
\end{equation}

Every $ \mathbf M(\ell-1) [a_i,j]$ is in $\{0, q,\dots, (p-1)q\}$. Therefore
\begin{equation}\label{mc23}
\sum_{i=1}^{s-1} v_i  \mathbf M(\ell-1) [a_i, j]  \in\{0, q,\dots, (p-1)q\}.
\end{equation}
By \eqref{mc21}, \eqref{mc22} and \eqref{mc23}, we have
\[
\begin{aligned}
\sum_{j=1}^{| \mathbf M(\ell)|}  \big|\sum_{i=1}^{s} v_i  \mathbf M(\ell) [a_i, j]\big|_n
& = \sum_{j=1}^{| \mathbf M(\ell-1)|}\frac{(p-1)pq}{2}  + \sum_{j=1}^{| \mathbf M(1)|}\big|v_s  \mathbf M(1)[1, j]\big|_n\\
& = \frac{(p^{\ell-1} - 1)| \mathbf M(1)|}{p - 1}\cdot\frac{(p-1)pq}{2} + \omega(1)\\
& = \begin{cases}
p^{\ell}q, & \text{if $p>2$}
 \\
2^{\ell - 1}q, & \text{if $p=2$}
 \end{cases}.
\end{aligned}
\]

\noindent
{\bf Case 3.} $\ell \geq 2$ and $a_s < \ell $.

Indeed, by the definition of $ \mathbf M_\ell$ and the rules of Cartesian product, we have
\begin{equation}\label{algorithm2}
\begin{aligned}
  \mathbf M(\ell ) & =  \mathbf M(\ell-1) \times  \mathbf A \cup \{0\}^{\ell-1} \times  \mathbf M(1)
\\ & = ( \mathbf M(\ell-2) \times  \mathbf A\cup \{0\}^{\ell-2} \times  \mathbf M(1)
 ) \times  \mathbf A \cup \{0\}^{\ell-1} \times  \mathbf M(1)
\\ & =  \mathbf M(\ell-2) \times  \mathbf A^2 \cup \{0\}^{\ell-2} \times  \mathbf M(1)\times \mathbf  A \cup \{0\}^{\ell-1} \times  \mathbf M(1)
\\ & \,\,\,\vdots
\\&= \mathbf M(a_s)\times  \mathbf A^{\ell-a_s}\bigcup_{t=a_s}^{\ell-1} \{0\}^{t} \times \mathbf  M(1) \times \mathbf  A^{\ell-t-1} .\end{aligned}
\end{equation}

Thus by \eqref{algorithm2} and $|\mathbf A^{\ell-a_s}| = p^{\ell - a_s}$, together with the result in {\bf Case 2.}, we can derive that
\[
\begin{aligned}
\sum_{j=1}^{|\mathbf M(\ell)|} \big|\sum_{i=1}^{s} v_i \mathbf M(\ell) [a_i, j]\big|_n
& = \sum_{j=1}^{|\mathbf M(a_s)|} \big|\sum_{i=1}^{s} v_i \mathbf M(a_s) [a_i, j]\big|_n\cdot p^{\ell-a_s}+0\\
& = \omega(a_s)\cdot p^{\ell-a_s}  = 
\begin{cases} 
p^{\ell}q, & \text{if $p>2$}
 \\
2^{\ell - 1}q, & \text{if $p=2$}
\end{cases}.
\end{aligned}
\]

\item
Since $ \mathbf A = - \mathbf A$ and $ \mathbf M(1) = - \mathbf M(1)$, by the definition of $\mathbf M(\ell)$, it follows that $\mathbf M(\ell) = - \mathbf M(\ell)$. Thus it is easy to infer that
\[
\begin{aligned}
\sum_{j = 1}^{|\mathbf M(\ell)|} \big| - \sum_{i = 1}^{s} v_i \mathbf M(\ell) [a_i, j]\big|_n = \sum_{j = 1}^{|\mathbf M(\ell)|} \big| \sum_{i = 1}^{s} v_i \mathbf M(\ell) [a_i, j] \big|_n.
\end{aligned}
\]
\end{enumerate}
\end{proof}

We need the following result which is a straightforward consequence of \cite[Lemma 1]{geroldinger1992Dav} and we omit the similar proof here. 

\begin{lem}\label{zhihe}
Let $G=C_{n_1}\oplus C_{n_2}\oplus \dots \oplus C_{n_r}$
with $1<n_1|n_2\dots |n_r$. Let $H_x= \oplus_{i\in I_x} C_{n_{i}}$, where $x\in[1,z]$, $z\in \mathbb N_+$, $\varnothing\neq I_x \subsetneq [1,r]$ and $I_{x} \cap I_{y}=\varnothing$ for any $x,y\in[1,z]$. Then
\[\Dav(G)-\Dav^*(G)\ge\sum_{x=1}^{z}(\Dav(H_x)-\Dav^*(H_x)).\]
\end{lem}

\section{On non-dispersive sequences over $C_n^r$}

In this section, we will construct long non-dispersive sequences by $\mathbf M(\ell)$'s. 
 
\begin{thm}\label{uniquelength}
Let $G = C_n^r$, where $r\in \mathbb N_+$ and $n\in \mathbb N_{\geq 2}$, and let $p$ be a prime divisor of $n$. If $\ell\in \mathbb N_+$ such that $r\ge\theta(\ell;p)\ge 1$, then there exists a sequence $S$ over $G$ of length 
$$|S|=(n-1)r+(p-1)\ell=\Dav^*(G)+(p-1)\ell-1,$$ 
such that every nonempty zero-sum subsequence $T$ of $S$ is length of $$|T|=\omega(\ell;n,p).$$  
\end{thm} 

\begin{proof}
Set $n=pq$, where $q\in[1,n]$.
\item
\noindent
{\bf Case 1.} $p>2$.

It follows from $r \ge \theta(\ell;p) \ge 1$ that $\ell\ge 1$. Let
$$
\mathbf E(\ell) = 
\bigcup_{t=0}^{\ell-1} {\{0\}}^t\times\{q,(p-1)q\}\times{\{0\}}^{\ell-t-1}$$
 and $$\mathbf F(\ell) = 
  \bigcup_{t=0}^{\ell-1}{\{0\}}^t\times\{1\}\times{\{0\}}^{\ell-t-1}.$$ 
Let
$
\mathbf W(\ell) = \mathbf M(\ell)
\backslash \mathbf E(\ell)
 \cup \mathbf F(\ell)$. 
Thus by Proposition \ref{PropM}, we have  $|\mathbf W(\ell)| = |\mathbf M(\ell)|  - 2\ell + \ell= \theta(\ell).$

List the elements of $\mathbf W(\ell)$ in some fixed but arbitrary order. Let $\mathbf W(\ell) [i, j]$ denote the $i$-th entry of the $j$-th element of $\mathbf W(\ell)$. For any indices  $$1 \le a_1 < \dots < a_s\le \ell \text{ and }v_i\in[1,p - 1] \text{ with }i\in[1,s],$$ also by Proposition \ref{PropM} and $n=pq$, we have \\
\begin{equation}\label{len}
\begin{aligned}
\sum_{j=1}^{|\mathbf W(\ell)|} &\big| \sum_{i=1}^{s} v_i \mathbf W(\ell) [a_i, j]\big|_n \\
=& \sum_{j=1}^{|\mathbf M(\ell)|} \big|  \sum_{i=1}^{s} v_i \mathbf M(\ell) [a_i,j] \big|_n  - \sum_{i=1}^{s}(|v_iq|_n + |v_i(p-1)q|_n) + \sum_{i=1}^{s}|v_i|_n\\
= &\omega(\ell) - \sum_{i=1}^{s}n + \sum_{i=1}^{s}(n-v_i)
=  \omega(\ell)-\sum_{i=1}^{s}v_i.
\end{aligned}
\end{equation}
Let $C_n^r =  \oplus_{j = 1}^r \langle e_{j}\rangle$ with $\ord(e_{j}) = n$ for each $j\in[1,r]$. By
 $r\ge\theta(\ell;p)$, we can set $$x_b = \sum_{j = 1}^{\theta(\ell)}\mathbf W(\ell)[b,j]\cdot e_j, \text{ where } b\in[1,\ell],$$
and let sequence \[
S = \prod_{j = 1}^{r} e_j^{n - 1} \prod_{b=1}^{\ell}x_b^{p-1}.
\]
Suppose that $S_1$ is a nonempty zero-sum subsequence of $S$. If $x_b$ does not occur in $S_1$ for any $b\in[1,\ell]$, then $S_1$ is zero-sum free. Thus, for any indices $1 \le a_1 < \dots < a_s\le \ell$ and any $v_i\in[1,p - 1]$ with $i\in[1,s]$, we set
\[
S_1 = \prod_{j = 1}^{r} e_j^{u_j} \prod_{i=1}^{s}x_{a_i}^{v_i},
\]
where $u_j\in[0,n-1]$. Since $S_1$ is zero-sum, we have
\[
 u_j = \big|n - \sum_{i=1}^{s} v_i \mathbf W(\ell) [a_i, j]\big|_n,\, j\in[1,\theta(\ell)],
\]
and $u_j = 0$ for $j>\theta(\ell)$.
Thus, together with \eqref{len} and Proposition \ref{PropM}, we obtain that
\[
|S_1| = \sum_{i=1}^{s}v_i + \sum_{j=1}^{|\mathbf W(\ell)|} \big|n - \sum_{i=1}^{s} v_i \mathbf W(\ell) [a_i, j]\big|_n = \omega(\ell),
\]
which completes the proof  of this lemma in  Case 1.

\item
\noindent
{\bf Case 2.} $p=2$.

It follows from $r\ge\theta(\ell;2)\ge 1$ that $\ell\ge 2$. 

Suppose that $r\ge 4$ and thus $\ell \ge 3$. Let 
$$\mathbf E(\ell) = 
\bigcup_{t=0}^{\ell-1} {\{0\}}^t\times\{q\}\times{\{0\}}^{\ell-t-1},$$
$$\mathbf F(\ell) =
\bigcup_{t=0}^{\ell-2} {\{0\}}^t\times\{q\}\times\{q\}\times{\{0\}}^{\ell-t-2},$$
$$\mathbf H(\ell) =
\bigcup_{t=0}^{\ell-2} {\{0\}}^t\times\{1\}\times\{q\}\times{\{0\}}^{\ell-t-2},$$
$$\mathbf I(\ell) = \{q\}\times{\{0\}}^{\ell-2}\times\{q\}$$
and 
$$\mathbf J(\ell) = \{q\}\times{\{0\}}^{\ell-2}\times\{1\}.$$
Let 
\begin{equation}\label{mwp2}
\mathbf W(\ell) = \mathbf M(\ell)\backslash \mathbf E(\ell)\backslash \mathbf F(\ell) \cup \mathbf H(\ell) \backslash \mathbf I(\ell) \cup \mathbf J(\ell).
\end{equation}
Thus by Proposition \ref{PropM}, we have  $|\mathbf W(\ell)| = |\mathbf M(\ell)|  - \ell -(\ell-1) + (\ell -1)-1 +1= \theta(\ell).$ Let  
$$\mathbf U(\ell)= \mathbf M(\ell)\big\backslash \big(\bigcup_{t=0}^{\ell-1} {\{0\}}^t\times\{q\}\times{\{0\}}^{\ell-t-1}\big).$$ 
By \eqref{mwp2}, for each $z\in [1,\ell]$, we can just change exactly one element $\mathbf U(\ell)[z,j_z]$ of $\mathbf U(\ell)$ from $q$ to $1$, to obtain  $\mathbf W(\ell)$. Also it should satisfy that, for all $\mathbf W(\ell)[x,j_z]$ with $z \ne x\in [1,\ell]$, there exists exactly one element $q$  and the others are $0$, and if $z_1 \ne z_2$, then $j_{z_1} \ne j_{z_2}$, where $z_1,z_2\in[1,\ell]$.

Hence, let indices $1 \le a_1 < \dots < a_s\le \ell$, for any $z\in \{a_1,\dots,a_s\}$, then either $$\sum_{i=1}^{s} \mathbf  U(\ell) [a_i, j_z]=q\text{ and }\sum_{i=1}^{s}  \mathbf W(\ell) [a_i, j_z] = 1 ,$$
or
$$\sum_{i=1}^{s} \mathbf  U(\ell)[a_i, j_z]=2q\text{ and } \sum_{i=1}^{s} \mathbf  W(\ell) [a_i, j_z] = q+1 .$$
So we have $$ \big|-  \sum_{i=1}^{s} \mathbf  W(\ell) [a_i, j_z]\big|_n - \big| - \sum_{i=1}^{s}  \mathbf U(\ell) [a_i, j_z]\big|_n  = q-1.$$
Together with  Proposition \ref{PropM} and $n=2q$, we have 
\begin{equation}
\begin{aligned}
\sum_{j=1}^{|\mathbf W(\ell)|} \big| n&- \sum_{i=1}^{s} \mathbf  W(\ell) [a_i, j]\big|_n \\
= &\sum_{j=1}^{|\mathbf W(\ell)|} \big| - \sum_{i=1}^{s}  \mathbf W(\ell) [a_i, j]\big|_n \\
 = &\sum_{j=1}^{|\mathbf M(\ell)|} \big|-  \sum_{i=1}^{s}  \mathbf M(\ell) (a_i,j) \big|_n  - \sum_{i=1}^{s}|-q|_n + \sum_{z\in \{a_1,\dots,a_s\}} (q-1)\\
 = &\sum_{j=1}^{|\mathbf M(\ell)|} \big|  \sum_{i=1}^{s}   \mathbf M(\ell) [a_i,j] \big|_n -  {s}q  + s(q-1)
=  \omega(\ell)-s.
\end{aligned}
\end{equation}

Suppose that $\ell=2$, let $\mathbf W(2)=\{(1,1)\}$. It is clear that  $|\mathbf W(2)|=\theta(2)=1$, and $\sum_{j=1}^{|\mathbf W(2)|} \big| n- \sum_{i=1}^{s} \mathbf  W(2) [a_i, j]\big|_n = \omega(2)-s$ for any indices  $1 \le a_1 < \dots < a_s\le \ell$.

Then by the similar proof in Case 1, we complete the proof.
\end{proof}

\begin{defi} (\cite{GZZ2015}) Define $\mathsf{disc}(G)$ to be the smallest positive integer $t$, such that every sequence over $G$ of length at least $t$ has two nonempty zero-sum subsequences of distinct lengths.
\end{defi}

By Theorem \ref{uniquelength}, we can derive the following corollary immediately.

\begin{cor}\label{disc}
Let $G = C_n^r$, where $r\in \mathbb N_+$ and $n\in \mathbb N_{\geq 2}$, and let $p$ be a prime divisor of $n$. If $\ell\in \mathbb N_+$ such that $r\in [\theta(\ell),\theta(\ell+1))$, then $\mathsf{disc}(G)\ge (n - 1)r + (p-1)\ell + 1$.
\end{cor}
Note that, for $n=2$, the above bound equals $\mathsf{disc}(G)$ (see \cite[Theorem 1.3]{GZZ2015}).

\section{On the lower bounds of $\Dav(G)$ }

By Lemma \ref{nondisp2dg} we connect the lower bounds of $\Dav(G)$ to special non-dispersive sequences. This lemma is a crucial one to this paper.

\begin{lem}\label{nondisp2dg}
 Let  $G=G_1\oplus \dots \oplus G_t \oplus C_m$, where $t\in \mathbb N_{+}$, $m\in \mathbb N_{\ge 2}$, and $G_1,\dots,G_t$ are finite abelian groups. For every $i\in[1,t]$, let $S_i$ be a non-dispersive sequence over $G_i$ which only contains zero-sum subsequences of length $x_i$. If $y=\sum_{i=1}^t \gcd(x_i,m)<m$, then $\Dav(G) \ge \sum_{i=1}^t |S_i|+m-y$.
\end{lem}
\begin{proof}
By results from the elementary number theory, for every $x_i$ with $i\in[1,t]$, there exists a $u_i\in[1,m-1]$ such that $|x_i u_i|_m=\gcd(x_i,m)$. 
Let $C_m=\langle e\rangle$. 
Consider the following sequence 
$$S=(S_1+u_1e)(S_2+u_2e)\dots (S_t +u_te) e^{m-y-1}.$$
Suppose that $S$ has a non-empty zero-sum subsequence $T$, and $$T=T_1T_2\dots T_t e^z \text{ with }T_i|(S_i+u_ie)\text{, }i\in[1,t]\text{ and }0\le z\le m-y-1.$$ 
We observe that the $S_i$'s and $e$ are independent and $S_i$ only contains zero-sum subsequences of length $x_i$. Thus  $|T_i|=x_i$ or $|T_i|=0$, for $i\in[1,t]$.
And the sum of $T$ is $ve$, where 
$$v=||T_1|u_1+|T_2|u_2+\dots+|T_t|u_t+z|_m.$$ 
Since $T$ is non-empty and 
$$\begin{aligned}
|x_1u_1|_m&+|x_2u_2|_m+\dots + |x_tu_t|_m+z\\
=&\sum_{i=1}^t \gcd(x_i,m)+z=y+z\le m-1,
\end{aligned}$$
it follows that $0<v<m$ and thus $T$ is not zero-sum. This contradicts the definition of $T$. Thus $S$ is zero-sum free and $\Dav(G)\ge |S|+1=  \sum_{i=1}^t |S_i|+m-y$.
\end{proof}

By Lemma \ref{nondisp2dg},  Theorem \ref{uniquelength} and Lemma \ref{zhihe}, we are able to construct long zero-sum free sequences over general abelian groups. Next, we would like to provide Theorem \ref{lzfs} and Corollary \ref{est} to easily estimate the growth of $\Dav(G)-\Dav^*(G)$ for large $r$ and  $\exp(G)$.

\begin{rem}\label{guaran}
Let  $G=C_n^r\oplus C_{kn}$ be a non-$p$-group with $n,k\in \mathbb N_{\ge 2}$. Then there exist $p$ and $k_1$ such that $p$ be a prime divisor of $n$, $k_1\in \mathbb N_{\geq 2}$ be a divisor of $k$ with $\gcd(p,k_1)=1$. We use this remark to guarantee that the result in Theorem \ref{lzfs} is not vacuous.
\end{rem}

\begin{proof}
 If $n=p^t>1$ is a prime power, since $G$ is a non-$p$-group,  there exists   $1<k_1 |k$ with $\gcd(p,k_1)=1.$ If $n$ has at least two distinct prime factors $p_1$ and $p_2$. Consider a prime factor $p_3$ of $k$, then either $\gcd(p_1,p_3)=1$ or $\gcd(p_2,p_3)=1$. Thus the existence is proved.

\end{proof}

\begin{thm}\label{lzfs}
Let  $G=C_n^r\oplus C_{kn}$ be a non-$p$-group with $n,k\in \mathbb N_{\ge 2}$. Let $p$ be a prime divisor of $n$, $k_1\in \mathbb N_{\geq 2}$ be a divisor of $k$, with $\gcd(p,k_1)=1$, and $kn=k_1m$ for some $m\in \mathbb N$. If $\ell\in \mathbb N_+$ and $t\in [1,k_1-1]$ with $r \ge t \theta(\ell)\ge 1$, then $$
\Dav(G) \ge \Dav^*(G) + t(p - 1)\ell -  tm.$$
\end{thm}

\begin{proof}
Let $(e_1,\dots,e_r)$ be a basis of $C_n^r$ with $\ord(e_1)=\dots=\ord(e_r)=n$.
Let \[
G_j = \oplus_{i = 1+(j-1)\theta(\ell) }^{j \theta(\ell)} \langle e_{i}\rangle,\text{ where }j\in[1,t-1],
\]
and let $G_{t}=\oplus_{i=1+(t-1)\theta(\ell)}^{r} \langle e_{i}\rangle$. By Theorem \ref{uniquelength}, there exists a sequence $S_j$ over each $G_j$ with $$|S_j|=\Dav^*(G_j)-1+(p-1)\ell,$$ which only contains zero-sum subsequences of a unique length $\omega(\ell)$. Hence, by $\gcd(p,k_1)=1$, we have 
$$\begin{aligned}
\gcd(\omega(\ell),kn)&\le\gcd(p^{\ell-1}n, kn)=n\gcd(p^{\ell-1},k)\\
&=n\gcd\left(p^{\ell-1},\frac{k}{k_1}\right)\le \frac{nk}{k_1}= m.\end{aligned}$$
And $\sum_{j=1}^{t}\gcd(\omega(\ell),kn)=tm<kn$. By Lemma \ref{nondisp2dg}, it follows that $$\begin{aligned}\Dav(G) &\geq \sum_{j=1}^t|S_j|+kn-\sum_{j=1}^t\gcd(\omega(\ell),kn)\\
 &\geq \sum_{j=1}^t|S_j|+kn-mt=\Dav^*(G) + ((p - 1)\ell - m)t.
\end{aligned}$$
\end{proof}

\begin{cor}\label{est}
Let  $G=C_n^r\oplus C_{kn}$ be a non-$p$-group with $n,k\in \mathbb N_{\ge 2}$. Let $p$ be a prime divisor of $n$, $k_1\in \mathbb N_{\geq 2}$ be a divisor of $k$, with $\gcd(p,k_1)=1$, and $kn=k_1m$ for some $m\in \mathbb N$. For any integer $t \in [1, k_1-1]$, we have 

\begin{equation}\label{xiajie}
\Dav(G) > \Dav^*(G)+\frac{t(p-1)}{\log p}\log  r-t(p-1)(\log_p  t +1)-tm.
\end{equation}
\end{cor}
\begin{proof}
In Remark \ref{guaran}, we proved the existence of $p$ and $k_1$. For every $r\in \mathbb N_+$,  there exists an $\ell\in \mathbb N_+$ such that $\theta(\ell)\ge 1$ and  $r\in[t\theta(\ell),t\theta(\ell+1) )$. By the definition of $\theta(\ell)$, we have $\theta(\ell+1)<p^{\ell+1}.$ Thus $r<t\theta(\ell+1)<tp^{\ell+1}$. It follows that $\ell > \log_p \frac{r}{t} -1.$ By Theorem \ref{lzfs}, we have $$\begin{aligned}\Dav(G)&\ge \Dav^*(G)+t((p-1)\ell-m)\\
&>\Dav^*(G)+t\left(\left(p-1\right) \left(\log_p \frac{r}{t} -1\right)-m\right) \\
&=\Dav^*(G)+\frac{t(p-1)}{\log p}\log  r-t(p-1)(\log_p  t +1)-tm.\end{aligned}$$ 
\end{proof}

\begin{rem}\label{taolun}
Let  $G=C_n^r\oplus C_{kn}$ be a non-$p$-group with $n,k\in \mathbb N_{\ge 2}$. In Corollary \ref{est}, let $t=1$, we have 
\begin{equation}\label{xiajiet1}
\Dav(G) > \Dav^*(G) + (p - 1)\log_p r - m-p+1.
\end{equation}
 So $\Dav(G)-\Dav^*(G)$ grows at least logarithmically with respect to $r$. And this inequality does not depend on the size of $k_1$. That is to say, it can be $\Dav(G)-\Dav^*(G)>0$ for arbitrarily large exponent of $G$.

We have $\Dav(G) -\Dav^*(G)>t\left(\left(p-1\right) \left(\log_p \frac{r}{t} -1\right)-m\right)$ by Corollary \ref{est}. Fix $p$ and $m$. Let $r$ be larger than some constant, by \eqref{xiajiet1}, then there always exists $t\in [1,k_1-1]$ such that $\Dav(G) -\Dav^*(G)>0$. Let $t=c_1 r$, where $c_1\in (0,1)$ is a real number such that $\left(p-1\right) \left(\log_p \frac{r}{t} -1\right)-m>0$. Then for sufficiently large $k_1=k_1(r)$ such that $t\in [1,k_1]$, by Corollary \ref{est}, we always have $\Dav(G) -\Dav^*(G)>c_2 r$, where $c_2>0$ is a constant determined by $p$, $m$ and $c_1$. Note that $c_1$ is bounded by $p$ and $m$. See \eqref{jiea} for more information about $\frac{\Dav(G) -\Dav^*(G)}{r}$.

On the other hand, fix $n$ and $k$, for sufficiently large $r$, we can let $t = k_1-1$ and $p$ be as large as possible to get larger $\Dav(G)-\Dav^*(G)$ in \eqref{xiajie}. 
\end{rem}

Next, we give a general lower bound to   abelian non-$p$-groups and express the lower bound of  $\Dav(G)  - \Dav^*(G)$  by the rank and the exponent of $G$.  In Theorem \ref{gene}, we define $\log(0)=-\infty$ for the case of $|G|=m^r$.  

\begin{thm}\label{gene}
Let $G$ be a finite abelian non-$p$-group of rank $r\in \mathbb N_+$ and exponent $m\in \mathbb N_{\geq 2}$. Then
$$
\Dav(G)  \ge \Dav^*(G) + \max\{  \log_2   \log \frac{m^r}{|G|}- 2\log_2   \log \frac{m}{2}-m+\log_2\log 2+1,\,0\}.
$$
\end{thm}

\begin{proof}
$\Dav(G)\ge \Dav^*(G)$ is trivial.

Note that any abelian non-$p$-group $G$'s exponent $m\ge 6$. So $\log \log \frac{m}{2}>0$. If $|G|= m^r$, since we define that $\log(0)=-\infty$, the inequality in this theorem holds. 

Suppose that $|G|\ne m^r$ and $$G=C^{x_1}_{n_1}\oplus \dots \oplus C^{x_t}_{n_{t}}\oplus C^{x}_{m}$$ with $n_1| \dots |n_{t} |m$ and $1<n_1<\dots<n_t<m$.  
Let $x_a=\max\{x_i, i\in[1,t]\}$.   By Lemma \ref{zhihe},  \eqref{xiajiet1} and $\frac{p-1}{\log p}\ge \frac{1}{\log 2}$, we have
\begin{equation}\label{zuidaxa}
\Dav(G) > \Dav^*(G) + \log_2  {x_a} -  m+1.
\end{equation}

Since $m\ge 2 n_t\ge 2^2 n_{t-1} \ge \dots \ge 2^{t} n_1,$ we have $t\le \log_2 \frac{m}{n_1}$. 
Together with
$x_a t\ge x_1+\dots+x_t=r-x.$
We derive that 
\begin{equation}\label{xarx}
x_a\ge \frac{r-x}{\log_2 \frac{m}{n_1}}.
\end{equation}
By $$\frac{m^r}{|G|}=\frac{m^r}{n_1^{x_1}  n_2^{x_2}  \dots n_t^{x_t} m^x}\le \left(\frac{m}{n_1}\right)^{r-x},$$ we have $r-x\ge \log_{\frac{m}{n_1}}{\frac{m^r}{|G|}}.$ Together with  \eqref{xarx}, we have $$x_a\ge  \frac{\log_{\frac{m}{n_1}}\frac{m^r}{|G|}}{\log_2 \frac{m}{n_1}}
=\frac{\log \frac{m^r}{|G|} \log 2}{\log^2 \frac{m}{n_1}}.$$
Then by \eqref{zuidaxa}, it follows that 
\[\begin{aligned}
\Dav(G) &> \Dav^*(G) + \log_2 \frac{\log \frac{m^r}{|G|} \log 2}{\log^2 \frac{m}{n_1}} - m +1\\
&\ge \Dav^*(G) + \log_2   \log \frac{m^r}{|G|}- 2\log_2   \log \frac{m}{2}-m+\log_2\log 2+1.
\end{aligned}\]
Thus the theorem is proved. 
\end{proof}

So far, all the known groups $G$ with $\Dav(G)-\Dav^*(G)>0$ are non-$p$-groups  satisfying $|G|<\exp(G)^{r(G)}$. We would like to generalize this to a corollary as follows.

\begin{cor}\label{gwithe}
Given a non-prime power $m>0$. Let $G$ be  abelian groups with exponent $m$ and rank $r$, then for each $N>0$ there exists an $\varepsilon=\varepsilon(N;m)>0$ such that if $\frac{|G|}{m^{r}}<\varepsilon $, then $\Dav(G)-\Dav^*(G)>N$.
\end{cor}

\begin{proof}
This follows directly from Theorem \ref{gene}. 
\end{proof}

\begin{rem}
Let  $G=C_n^r\oplus C_{kn}$ be a non-$p$-group with $n,k\in \mathbb N_{\ge 2}$, we can consider the    {\sl small rank $r$} such that $\Dav(G)>\Dav^*(G)$. 
Theorem \ref{lzfs} shows that if $(p - 1)\ell -  m>0,$ then $\Dav(G)>\Dav^*(G).$ Thus, let $\ell= \lfloor \frac{m}{p-1} \rfloor+1.$
And $r=\theta ( \lfloor \frac{m}{p-1} \rfloor+1 )$ is a    small $r$ such that $\Dav(G)>\Dav^*(G)$.  
\end{rem}

The groups $G$ of {\sl small rank} with $\Dav(G)>\Dav^*(G)$ were viewed as ``{\sl the interesting groups}" on page $148$ in \cite{geroldinger1992Dav}.  We  give   following corollary about the small rank.

\begin{cor}
\begin{enumerate}[1)]
\item Let $G=C_p^r\oplus C_{kp}$ with  $p$ odd prime and $\gcd(p,k)=1$. If $r\ge 2p,$ then $\Dav(G)- \Dav^*(G) \ge  p-2>0.$ Thus 
\begin{equation}\label{jiea}
\sup_{\text{all finite abelian  group } G}\frac{\Dav(G)- \Dav^*(G) }{r}\ge \frac{1}{2}.\end{equation}

\item Let $G=C_2^r\oplus C_{2^t k}$ with $k>2$ odd and integer $t\ge 1$. If $r\ge 2^{2^t+1}-2^t-2$, then $\Dav(G)\ge \Dav^*(G) + 1.$
\end{enumerate}
\end{cor}
\begin{proof}
\begin{enumerate}[1)]
\item  Let $\ell=2$, then $\theta(\ell)=2p.$ By Theorem  \ref{lzfs}, if $r\ge 1\cdot \theta(\ell)=2p,$ then $\Dav(G)- \Dav^*(G) \ge  (p-1)\ell-p=p-2>0 .$

\item Let $\ell=2^t+1$ and $p=2$, then $\theta(\ell)=2^{2^t+1}-2-2^t.$  By Theorem \ref{lzfs}, if $r\ge  \theta(\ell),$ then $\Dav(G)\ge \Dav^*(G) + \ell-2^t=\Dav^*(G)+ 1 .$
\end{enumerate}
\end{proof}

In particular, let $G=C_2^r\oplus C_{2k}$ with  $k\ge 3$ odd. If $r\ge 4,$ then $\Dav(G)- \Dav^*(G) \ge  1.$ Note that for abelian group $G=C_2^4\oplus C_{2k}$ with odd $k\ge 70$, it is proved that $\Dav(G)= \Dav^*(G) + 1$ (see  \cite{SavChen2012}). In addition, it is interesting to determine $\sup\frac{\Dav(G)- \Dav^*(G) }{r}$, where $G$ runs over all finite abelian groups.

\section{Concluding remarks}

\noindent
{\bf Open problem.} {\sl By Lemma \ref{zhihe}, a natural question occurs. What are the groups $G$, with  the invariant factor decomposition
$$G=C_{n_1}\oplus C_{n_2}\oplus \dots \oplus C_{n_r}\text{ with } 1<n_1|n_2\dots |n_r,$$ 
 such that there do not exist  groups 
$$H_x= \oplus_{i\in I_x} C_{n_{i}},\text{ with }\varnothing\neq I_x \subsetneq [1,r]\text{ and }I_{x} \cap I_{y}=\varnothing \text{  for any } x,y\in[1,z],$$
satisfying that
$\Dav(G)-\Dav^*(G)=\sum_{x=1}^{z}(\Dav(H_x)-\Dav^*(H_x)).$
}

\bigskip
\noindent
{\bf Acknowledgments.} {The author would like to thank all the anonymous referees for their careful reading and many valuable suggestions on improving the paper. We also would like to thank Prof. W. Gao for his helpful comments and suggestions. This work was supported by the 973 Program of China (Grant No.2013CB834204), the PCSIRT Project of the Ministry of Science and Technology, and the National Science Foundation of China (Grant No.11671218). }

\end{document}